\pdfoutput=1
\RequirePackage{ifpdf}
\ifpdf 
\documentclass[pdftex]{sigma}
\else
\documentclass{sigma}
\fi

\numberwithin{equation}{section}

\newtheorem{Theorem}{Theorem}[section]
\newtheorem{Proposition}[Theorem]{Proposition}
 { \theoremstyle{definition}
\newtheorem{Remark}[Theorem]{Remark} }

\begin{document}


\renewcommand{\thefootnote}{$\star$}

\newcommand{\arXivNumber}{1511.07056}

\renewcommand{\PaperNumber}{020}

\FirstPageHeading

\ShortArticleName{Orthogonal Polynomials on the Unit Ball and Fourth-Order Partial Dif\/ferential Equations}

\ArticleName{Orthogonal Polynomials on the Unit Ball \\
and Fourth-Order Partial Dif\/ferential Equations\footnote{This paper is a~contribution to the Special Issue
on Orthogonal Polynomials, Special Functions and Applications.
The full collection is available at \href{http://www.emis.de/journals/SIGMA/OPSFA2015.html}{http://www.emis.de/journals/SIGMA/OPSFA2015.html}}}

\Author{Clotilde MART\'{I}NEZ and Miguel A.~PI\~NAR}
\AuthorNameForHeading{C.~Mart\'{\i}nez and M.A.~Pi\~nar}
\Address{Departamento de Matem\'atica Aplicada, Universidad de Granada, 18071 Granada, Spain}
\Email{\href{mailto:clotilde@ugr.es}{clotilde@ugr.es}, \href{mailto:mpinar@ugr.es}{mpinar@ugr.es}}
\URLaddress{\url{http://www.ugr.es/local/mpinar}}

\ArticleDates{Received November 24, 2015, in f\/inal form February 18, 2016; Published online February 23, 2016}	

\Abstract{The purpose of this work is to analyse a family of mutually
orthogonal polynomials on the unit ball with respect to an inner
product which includes an additional term on the sphere.
First, we will get connection formulas relating
classical multivariate orthogonal polynomials on the ball
with our family of orthogonal polynomials.
Then, using the representation of these polynomials in terms of
spherical harmonics, algebraic and dif\/ferential
properties will be deduced.}


\Classification{33C50; 42C10}

\renewcommand{\thefootnote}{\arabic{footnote}}
\setcounter{footnote}{0}

\vspace{-2mm}

\section{Introduction}

In 1940 (see \cite{Krall40}), H.L.~Krall classif\/ied orthogonal polynomials
satisfying fourth-order dif\/ferential equations. Apart from the classical orthogonal
polynomials we can only f\/ind three other families, the now called Jacobi-type, Legendre-type and
Laguerre-type polynomials. The corresponding orthogonality measures are some
particular cases of classical measures modif\/ied by the addition of a Dirac delta
at the end points of the interval of orthogonality. Further details on those families
of polynomials were given later by A.M.~Krall~\cite{Krall81},
L.L.~Littlejohn \cite{LL82} and T.H.~Koornwinder~\cite{K84}.

Our main goal in this work is to extend the study of orthogonal polynomials
satisfying fourth-order dif\/ferential equations to a~multidimensional context.
In particular, we consider a~nonclassical weight function supported on
the $d$-dimensional unit ball and we will show that the corresponding orthogonal
polynomials will satisfy a fourth-order partial dif\/ferential equation.

In our study, classical orthogonal polynomials on the unit ball $\mathbb{B}^d$ of $\mathbb{R}^d$
play an essential role. They are orthogonal with respect to the inner product
\begin{gather*}
 \langle f, g\rangle_\mu : = \frac{1}{\omega_\mu}\int_{\mathbb{B}^d} f(x) g(x) W_\mu(x) dx,
\end{gather*}
where $W_\mu(x)$ is the weight function given by
\begin{gather*}
W_\mu(x): = \big(1-\|x\|^2\big)^{\mu-1/2}
\end{gather*}
on $\mathbb{B}^d$, $\mu > -1/2$, and
$\omega_\mu$ is a normalizing constant such that $\langle 1 ,1 \rangle_\mu = 1$.
The condition $\mu > -1/2$ is necessary for the inner product to be well def\/ined.

In the present paper, we will consider orthogonal polynomials with respect to the inner product
\begin{gather} \label{eq:main-ip}
 \langle f ,g \rangle_{\mu}^{\lambda}
 = \frac{1}{\omega_\mu} \int_{\mathbb{B}^d} f(x) g(x) W_\mu(x) dx +
 \frac{\lambda}{\sigma_{d-1}} \int_{\mathbb{S}^{d-1}} f(x) g(x) d\sigma,
\end{gather}
where $\lambda>0$, $d\sigma$ \looseness=-1 denotes the surface measure on the unit sphere $\mathbb{S}^{d-1}$,
and~$\sigma_{d-1}$ denotes the area of~$\mathbb{S}^{d-1}$.
The presence of the extra spherical term in~\eqref{eq:main-ip} changes
the importance of the values on the sphere of a given function~$f$ when
approximating by means of the corresponding Fourier series.

Using spherical harmonics {and polar coordinates}
we shall construct explicitly a sequence of mutually orthogonal
polynomials with respect to $ \langle f ,g \rangle_{\mu}^{\lambda}$, which depends
on a family of polynomials of one variable. The latter ones can be expressed
in terms of Jacobi polynomials. A~similar construction was used in~\cite{DFLPP2015} and~\cite{PPX2013} to obtain a sequence of mutually orthogonal
polynomials with respect to a Sobolev inner product. In~\cite{DFLPP2015}, the
inner product includes the outward normal derivatives over the sphere and
aside from getting the connection with classical polynomials on the ball,
some asymptotic properties are obtained. On the other hand, the Sobolev inner product
in~\cite{PPX2013} was def\/ined in terms of the standard gradient operator.

Then, in the case $\mu = 1/2$, we will
show that the real space of polynomials
in $d$ variables admits a mutually orthogonal base whose elements are eigenfunctions
of a fourth-order linear partial dif\/ferential operator where the coef\/f\/icients
{are independent of} the degree of the polynomials.
{The corresponding eigenvalues depend on the dif\/ferent indices of the polynomials.}
{Thus, we get a} non-trivial example of a family of
multivariate polynomials
{orthogonal with respect a measure and}
satisfying a fourth-order dif\/ferential equation.

Multivariate polynomials satisfying fourth-order dif\/ferential
equations were constructed previously in \cite[Section 6.2]{Iliev2011}. In this work,
the author uses also polar coordinates and spherical harmonics to construct
a basis of polynomials which are eigenfunctions of a dif\/ferential operator
where the polynomial coef\/f\/icients are independent of the indexes
of the polynomials. The eigenvalues in \cite{Iliev2011} depend only on
the total degree of the polynomials. However, the polynomials are
orthogonal with respect to an inner product involving the spherical Laplacian,
that is, a~Sobolev product (see \cite[equation~(6.8)]{Iliev2011}).
Our approach is dif\/ferent, we start from the standard inner product def\/ined
in \eqref{eq:main-ip} and we f\/ind a fourth-order partial dif\/ferential
operator having the orthogonal polynomials as eigenfunctions.

The paper is organized as follows. In the next section, we state some materials on orthogonal
polynomials on the unit ball and spherical harmonics that we will need later. In Section~\ref{section3},
using spherical harmonics, a basis of mutually orthogonal polynomials associated to~\eqref{eq:main-ip} is constructed. In Section~\ref{section4}, in the case $\mu= 1/2$, we consider the dif\/ferential
properties of the polynomials in the radial parts of the polynomials in the previous section. Finally, in
Section~\ref{section5}, we show that these polynomials are eigenfunctions of a fourth-order dif\/ferential
operator where the coef\/f\/icients do not depend on the degree of the polynomials.

\section{Classical orthogonal polynomials on the ball}\label{section2}

In this section we describe background materials on orthogonal polynomials and spherical harmonics.
The f\/irst subsection collects properties on the Jacobi polynomials that
we shall need later. The second subsection is devoted to the Jacobi-type
orthogonal polynomials, associated to a weight function obtained as a
linear combination of the classical Jacobi weight function and a delta
function at $1$. The third subsection recalls the basic results on spherical harmonics
and classical orthogonal polynomials on the unit ball.

\subsection{Classical Jacobi polynomials}

We collect some properties of the classical Jacobi polynomials $P_n^{(\alpha,\beta)}(t)$,
all of them can be found in \cite{Sz}.
For $\alpha, \beta > -1$, these
polynomials are orthogonal with respect to the Jacobi inner product
\begin{gather*}
\left(f,g\right)_{\alpha,\beta}=\int_{-1}^1f(t) g(t) (1-t)^\alpha(1+t)^{\beta} dt.
\end{gather*}
The Jacobi polynomial
$P_n^{(\alpha, \beta)}(t)$ is normalized by
\begin{gather*} 
P_n^{(\alpha,\beta)}(1) = \binom{n+\alpha}{n} = \frac{(\alpha+1)_n}{n!}.
\end{gather*}

The derivative of a Jacobi polynomial is again a Jacobi polynomial
\begin{gather}\label{derJ}
\frac{d}{d t} P_n^{(\alpha,\beta)}(t) = \frac{n+\alpha + \beta+1}{2} P_{n-1}^{(\alpha+1,\beta+1)}(t).
\end{gather}
The polynomials $P_n^{(\alpha, \beta)}(t)$ satisfy the second-order dif\/ferential equation
\begin{gather}\label{diffeqJ}
\big(1-t^2\big)y''(t) + [\beta - \alpha -(\alpha + \beta + 2)t] y'(t) = -n(n+\alpha+\beta+1)y(t),
\end{gather}
where $y(t) = P_n^{(\alpha,\beta)}(t)$.

\subsection{Orthogonal polynomials on the unit ball and spherical harmonics}

For a multi-index $\nu\in\mathbb{N}_0^d$, $\nu = (\nu_1, \ldots, \nu_d)$ and
$x = (x_1, \ldots, x_d)$, a monomial in the variables
$x_1, \ldots, x_d$ is a product
\begin{gather*}
x^{\nu} = x_1^{\nu_1} \cdots x_d^{\nu_d}.
\end{gather*}
The number $|\nu| = \nu_1 + \cdots + \nu_d$ is called the total degree of $x^{\nu}$.
A polynomial $P$ in $d$ variables is a f\/inite linear combination of monomials.

Let $\Pi^d$ denote the space of polynomials in~$d$ real variables. For a given
nonnegative integer~$n$,
let $\Pi_n^d$ denote the linear space of polynomials in several variables of (total) degree
at most~$n$ and let~$\mathcal{P}_n^d$ denote the space of homogeneous polynomials of degree~$n$.

The unit ball and the unit sphere in $\mathbb{R}^d$ are
denoted, respectively, by
\begin{gather*}
\mathbb{B}^d :=\big\{x\in \mathbb{R}^d\colon \|x\| \leqslant 1\big\} \qquad \textrm{and} \qquad \mathbb{S}^{d-1}:=\big\{\xi\in \mathbb{R}^d\colon \|\xi\| = 1\big\}.
\end{gather*}
where $\|x\|$ denotes as usual the Euclidean norm of $x$.

For $\mu \in \mathbb{R}$, the weight function $W_\mu(x) = (1-\|x\|^2)^{\mu-1/2}$ is integrable on the unit ball if $\mu > -1/2$. Let us denote the normalization
constant of $W_\mu$ by ${\omega_\mu}$,
\begin{gather}\label{omegamu}
\omega_\mu := \int_{{\mathbb{B}^d}} W_\mu(x) dx = \frac{\pi^{d/2}\Gamma(\mu+1/2)}{\Gamma(\mu + (d+1)/2)},
\end{gather}
and consider the inner product
\begin{gather*}
 \langle f,g \rangle_\mu = \frac{1}{\omega_\mu} \int_{{\mathbb{B}^d}} f(x) g(x) W_\mu(x) dx,
\end{gather*}
which is normalized so that $\langle 1,1\rangle_\mu = 1$.

A polynomial $P \in \Pi_n^d$ is called \textit{orthogonal} with respect to $W_\mu$ on the ball if
$\langle P, Q\rangle_\mu =0$ for all $Q \in \Pi_{n-1}^d$, that is, if it is orthogonal to all polynomials
of lower {degree}. Let $\mathcal{V} _n^d(W_\mu)$ denote the space of orthogonal polynomials of total
degree $n$ with respect to $W_\mu$.

For $n\ge 0$, let $\{P^n_{\nu}(x)\colon |\nu|=n\}$ denote a basis of $\mathcal{V}_n^d(W_\mu)$.
Notice that every element of $\mathcal{V} _n^d(W_\mu)$ is orthogonal to polynomials of lower degree. If the
elements of the basis are also orthogonal to each other, that is,
$\langle P_\nu^n, P_\eta^n \rangle_\mu=0$ whenever $\nu \ne \eta$,
we call the basis \textit{mutually orthogonal}. If, in addition,
$\langle P_\nu^n, P_\nu^n \rangle_\mu =1$, we call the basis \textit{orthonormal}.

Harmonic polynomials of degree $n$ in $d$-variables are polynomials in $\mathcal{P} _n^d$ that satisfy
the Laplace equation $\Delta Y = 0$, where
$
\Delta = \frac{\partial^2}{\partial x_1^2} + \cdots + \frac{\partial^2}{\partial x_d^2}
$
is the usual Laplace operator.

Let $\mathcal{H}_n^d$ denote the space of harmonic polynomials
of degree $n$. It is well known that
\begin{gather*}
 a_n^d: = \dim \mathcal{H}_n^d = \binom{n+d-1}{d-1} - \binom{n+d-3}{d-1}.
\end{gather*}
Spherical harmonics are the restriction of harmonic polynomials to the unit sphere. If $Y \in
\mathcal{H}_n^d$, {then} in spherical-polar {coordinates}
$x = r \xi$, $ r > 0$ and $\xi \in \mathbb{S}^{d-1}$, we get
\begin{gather*}
Y(x) = r^n Y(\xi),
\end{gather*}
so that $Y$ is uniquely determined by its restriction to the sphere. We shall also use
$\mathcal{H}_n^d$ to denote the space of spherical harmonics of degree $n$.

Let $d \sigma$ denote the surface measure and $\sigma_{d-1}$ denote the surface area,
\begin{gather} \label{sigma_d}
 \sigma_{d-1} := \int_{\mathbb{S}^{d-1}} d\sigma = \frac{2 \pi^{d/2}}{\Gamma(d/2)}.
\end{gather}
Spherical harmonics of dif\/ferent degrees are orthogonal with respect to the inner product
\begin{gather*}
 \langle f, g \rangle_{\mathbb{S}^{d-1}}: = \frac{1}{\sigma_{d-1}} \int_{\mathbb{S}^{d-1}} f(\xi) g(\xi) d\sigma(\xi).
\end{gather*}

If $Y(x)$ is a harmonic polynomial of degree $n$, by Euler's equation for homogeneous polynomials,
we deduce
\begin{gather*}
\langle x, \nabla\rangle Y(x) = \sum_{i=1}^d x_i \frac{\partial}{\partial x_i} Y(x) = n Y(x).
\end{gather*}

In spherical-polar coordinates $x = r \xi$, $r > 0$ and $\xi \in \mathbb{S}^{d-1}$,
the dif\/ferential operators $\Delta$ and~$\langle x, \nabla\rangle$ can be decomposed
as follows (cf.~\cite{DX12}):
\begin{gather}
\Delta = \frac{\partial^2}{\partial r^2} + \frac{d-1}{r}\frac{\partial}{\partial r} + \frac{1}{r^2}\Delta_0, \label{L-B-operator}\\
\langle x, \nabla\rangle = r \frac{\partial}{\partial r}. \label{x-nabla}
\end{gather}
The operator $\Delta_0$, which is the spherical part of the Laplacian,
is called the Laplace--Beltrami operator and it has spherical harmonics
as its eigenfunctions. More precisely, it holds (cf.~\cite{DX12})
\begin{gather}\label{eigen}
\Delta_0 Y (\xi) = -n(n+d-2)Y(\xi), \qquad \forall \, Y \in \mathcal{H}_n^d, \qquad \xi \in \mathbb{S}^{d-1}.
\end{gather}

In spherical-polar coordinates $x = r \xi$, $ r > 0$ and $\xi \in \mathbb{S}^{d-1} $, a mutually orthogonal
basis of $\mathcal{V}_n^d(W_\mu)$ can be given in terms of the Jacobi polynomials and spherical harmonics
(see, for instance,~\cite{DX14}).

\begin{Proposition} \label{prop-2.1}
For $n \in \mathbb{N}_0$ and $0 \leqslant k \leqslant n/2$, let $\{Y_\nu^{n-2k}(x)\colon 1\leqslant \nu\leqslant a_{n-2k}^d\}$ denote
an orthonormal basis for $\mathcal{H}_{n-2k}^d$. Define
\begin{gather}\label{baseP}
P_{k,\nu}^{n}(x) = P_{k}^{(\mu-\frac{1}{2}, \beta_k)}\big(2 \|x\|^2 -1\big) Y_\nu^{n-2k}(x),
\end{gather}
where $\beta_k = n-2k + \frac{d-2}{2}$.
Then the set $\{P_{k,\nu}^{n}(x)\colon 0 \leqslant k \leqslant n/2, \, 1 \leqslant \nu \leqslant a_{n-2k}^d \}$
is a mutually orthogonal basis of~$\mathcal{V} _n^d(W_\mu)$.
\end{Proposition}

It is known that orthogonal polynomials with respect to~$W_\mu$ are eigenfunctions of a second-order
dif\/ferential operator~$\mathcal{D}_\mu$. More precisely, we have
\begin{align} \label{eq:Bdiff}
 \mathcal{D}_\mu P = -(n+d) (n + 2 \mu - 1)P, \qquad \forall \, P \in \mathcal{V} _n^d(W_\mu),
\end{align}
where
\begin{gather*}
 \mathcal{D}_\mu := \Delta - \sum_{j=1}^d \frac{\partial}{\partial x_j} x_j \left[
 2 \mu - 1 + \sum_{i=1}^d x_i \frac{\partial }
 {\partial x_i} \right].
\end{gather*}

\section{An inner product on the ball with an extra spherical term}\label{section3}

Let us def\/ine the inner product
\begin{gather*} 
 \langle f ,g \rangle_{\mu}^{\lambda} = \frac{1}{\omega_\mu} \int_{\mathbb{B}^d} f(x) g(x) W_\mu(x) dx +
 \frac{\lambda}{\sigma_{d-1}} \int_{\mathbb{S}^{d-1}} f(\xi) g(\xi) d\sigma,
\end{gather*}
where $\lambda>0$, $d\sigma$ denotes the surface measure on the unit sphere $\mathbb{S}^{d-1}$,
$\sigma_{d-1}$ denotes the area of~$\mathbb{S}^{d-1}$ and~$\omega_\mu$ is the normalizing
constant~\eqref{omegamu}.

As a consequence of the central symmetry of the inner product, we can use a construction
analogous to~\eqref{baseP} to obtain a basis of mutually
orthogonal polynomials with respect to $ \langle f ,g \rangle_{\mu}^{\lambda}$. This time,
the radial parts constitute a family of polynomials in one variable related to Jacobi polynomials.

\begin{Theorem} \label{thm-3.1}
For $n \in \mathbb{N}_0$ and $0 \le k \le n/2$, let $\{Y_\nu^{n-2k}\colon 1\le \nu\le a_{n-2k}^d\}$ denote
an orthonormal basis for $\mathcal{H}_{n-2k}^d$.
Let $\beta_k = n -2k + (d-2)/2$
and let $q_{k}^{(\mu -1/2, \beta_k,\lambda)}(t)$ be the $k$-th orthogonal polynomial with respect to
\begin{gather} \label{eq:op1d}
(f,g)_{\mu-1/2,\beta_k}^{\lambda} = \frac{\sigma_{d-1}}{\omega_{\mu}}\frac{1}{2^{\mu+\beta_k+3/2}}
\int_{-1}^1 f(t)g(t)(1-t)^{\mu-1/2} (1+t)^{\beta_k} dt + \lambda f(1) g(1).
\end{gather}
Then the polynomials
\begin{gather*}
Q_{k,\nu}^{n}(x) = q_{k}^{(\mu-1/2, \beta_k,\lambda)}\big(2 \|x\|^2 -1\big) Y_\nu^{n-2k}(x),
\end{gather*}
with $1 \le k \le n/2$, $1 \le \nu \le a_{n-2k}^d$ constitute a mutually
orthogonal basis of $\mathcal{V}_n^d(W_\mu,\lambda)$ the linear space of orthogonal polynomials of
degree exactly $n$ with respect to $\langle \cdot, \cdot \rangle_{\mu}^{\lambda}$.
\end{Theorem}

\begin{proof}
The proof of this theorem uses the following well known identity
\begin{gather}\label{changevar}
 \int_{\mathbb{B}^d} f(x) dx = \int_0^1 r^{d-1}\int_{\mathbb{S}^{d-1}}f(r \xi) d\sigma (\xi) dr
\end{gather}
that arises from the spherical-polar coordinates $x=r \xi$, $\xi\in \mathbb{S}^{d-1}$.

In order to check the orthogonality, we need to compute the product
\begin{gather} \label{pr1}
\langle Q_{j,\nu}^n, Q_{k,\eta}^m \rangle_\mu^{\lambda} = \frac{1}{\omega_\mu} \int_{\mathbb{B}^d} Q_{j,\nu}^n(x) Q_{k,\eta}^m(x) W_\mu(x) dx + \frac{\lambda}{\sigma_{d-1}} \int_{\mathbb{S}^{d-1}} Q_{j,\nu}^n(\xi) Q_{k,\eta}^m(\xi) d\sigma(\xi).
\end{gather}

Let us start with the computation of the f\/irst integral
\begin{gather*}
I_1=\frac{1}{\omega_\mu} \int_{\mathbb{B}^d} Q_{j,\nu}^n(x) Q_{k,\eta}^m(x) W_\mu(x) dx.
\end{gather*}
Using polar coordinates, relation \eqref{changevar}, and the orthogonality of the
spherical harmonics we obtain
\begin{gather*}
I_1 = \frac{\sigma_{d-1}}{\omega_\mu} \int_0^1 q_j^{(\mu-1/2,\beta_j,\lambda)}\big(2r^2-1\big) q_k^{(\mu-1/2,\beta_k,\lambda)}\big(2r^2-1\big)\big(1-r^2\big)^{\mu-1/2} r^{n-2j+m-2k+d-1} dr
\\
\hphantom{I_1 =}{} \times \delta_{n-2j,m-2k} \delta_{\nu\eta}
\\
\hphantom{I_1}{} = \frac{\sigma_{d-1}}{\omega_\mu} \int_0^1 q_j^{(\mu-1/2,\beta_j,\lambda)}\big(2r^2-1\big) q_k^{(\mu-1/2,\beta_j,\lambda)}\big(2r^2-1\big) \big(1-r^2\big)^{\mu-1/2} r^{2(n-2j)+d-1} dr
\\
\hphantom{I_1}{} \times \delta_{n-2j,m-2k} \delta_{\nu\eta}.
\end{gather*}
Finally, the change of variables $t=2r^2-1$ moves the integral to the interval $[-1,1]$,
\begin{gather} \label{i1}
I_1 = \frac{1}{2^{\beta_j+\mu+3/2}} \frac{\sigma_{d-1}}{\omega_\mu} \int_{-1}^1 q_j^{(\mu,\beta_j,\lambda)}(t) q_k^{(\mu,\beta_j,\lambda)}(t) (1-t)^{\mu-1/2} (1+t)^{\beta_j} dt \times \delta_{n-2j,m-2k} \delta_{\nu\eta}.\!\!\!
\end{gather}
Let us now compute the second integral in \eqref{pr1},
\begin{gather}
I_2 = \frac{\lambda}{\sigma_{d-1}} \int_{\mathbb{S}^{d-1}} Q^n_{j,\nu}(\xi) Q^m_{k,\eta}(\xi) d\sigma(\xi)
\nonumber\\
\hphantom{I_2}{} = \frac{\lambda}{\sigma_{d-1}}
q_j^{(\mu-1/2,\beta_j,\lambda)}(1) q_k^{(\mu-1/2,\beta_k,\lambda)}(1)
 \int_{\mathbb{S}^{d-1}} Y^{n-2j}_{\nu}(\xi) Y^{n-2k}_{\eta}(\xi) d\sigma(\xi) \notag \\
\hphantom{I_2}{}= \lambda q_j^{(\mu-1/2,\beta_j,\lambda)}(1) q_k^{(\mu-1/2,\beta_k,\lambda)}(1)
 \delta_{n-2j,m-2k} \delta_{\nu,\eta}.\label{i2}
\end{gather}

To end the proof, we just have to put together \eqref{i1} and \eqref{i2} to get the value of \eqref{pr1}
in terms of the inner product~\eqref{eq:op1d} as
\begin{gather*}
\langle Q_{j,\nu}^n, Q_{k,\eta}^m \rangle_\mu^\lambda = \big(q_j^{(\mu-1/2,\beta_j,\lambda)}, q_k^{(\mu-1/2,\beta_k,\lambda)}\big)_{\mu-1/2,\beta_j}^\lambda \delta_{n-2j,m-2k} \delta_{\nu,\eta}.
\end{gather*}
And the result follows from the orthogonality of the univariate polynomial
$q_k^{(\mu-1/2,\beta_k,\lambda)}$.
\end{proof}

\begin{Remark}
Of course, Theorem~\ref{thm-3.1} can be formulated if we replace the weight function
$W_{\mu}(x)$ by any rotation invariant orthogonality
measure on the unit ball.
\end{Remark}

Following Koornwinder \cite[Theorem~3.1]{K84},
orthogonal polynomials with respect to~\eqref{eq:op1d} can be written in
terms of Jacobi polynomials as we show in the following theorem.

\begin{Theorem} \label{thm-3.2}
Let $\{q_{k}^{(\alpha, \beta,\lambda)}(t)\}_{k\geqslant0}$ be the sequence of polynomials defined by
\begin{gather} \label{poly_q}
q_{k}^{(\alpha, \beta,\lambda)}(t) = \left[a_k - (1+t)\frac{d}{dt}\right] P_{k}^{(\alpha, \beta)}(t),
\end{gather}
where
\begin{gather*} 
a_k = \frac{1}{\lambda} \frac{\Gamma(\alpha+\frac{d}{2}+1)}{\Gamma(\frac{d}{2})}\frac{k!}{(\alpha+1)_k} \frac{\Gamma(\beta+k+1)}{\Gamma(\alpha+\beta+k+1)} +
\frac{k (k+\alpha+\beta+1)}{\alpha + 1}.
\end{gather*}
Then they are orthogonal with respect to the inner product
\begin{gather} \label{eq:op1d_alpha}
(f,g)_{\alpha,\beta}^{\lambda} = \frac{\Gamma(\alpha+\frac{d}{2}+1)}{\Gamma(\frac{d}{2}) \Gamma(\alpha+1)}\frac{1}{2^{\alpha+\beta+1}}
\int_{-1}^1 f(t)g(t)(1-t)^{\alpha} (1+t)^{\beta} dt + \lambda f(1) g(1),
\end{gather}
and satisfy the normalization
\begin{gather*}
q_{k}^{(\alpha, \beta,\lambda)}(1) =
\frac{1}{\lambda} \frac{\Gamma(\alpha+\frac{d}{2}+1)}{\Gamma(\frac{d}{2})} \frac{\Gamma(\beta+k+1)}{\Gamma(\alpha+\beta+k+1)}.
\end{gather*}
\end{Theorem}

\begin{proof}
The polynomials def\/ined in \eqref{poly_q} are of exact degree $k$ and their orthogonality
can be deduced using the basis $\{(1-t)^j\}_{0 \leqslant j \leqslant k-1}$.
From \eqref{derJ} and the orthogonality of the Jacobi
polynomials we easily deduce
\begin{gather*}
\big(q_{k}^{(\alpha, \beta,\lambda)}(t),(1-t)^j\big)_{\alpha,\beta}^{\lambda} = 0, \qquad
j = 1, \ldots, k-1.
\end{gather*}
Now consider the case $j=0$. Then
\begin{gather*}
\big(q_{k}^{(\alpha, \beta,\lambda)},1\big)_{\alpha,\beta}^{\lambda} =
\frac{\Gamma(\alpha+\frac{d}{2}+1)}{\Gamma(\frac{d}{2}) \Gamma(\alpha+1)}\frac{1}{2^{\alpha+\beta+1}}
\int_{-1}^1 q_{k}^{(\alpha, \beta,\lambda)}(t) (1-t)^{\alpha} (1+t)^{\beta} dt + \lambda q_{k}^{(\alpha, \beta,\lambda)}(1) \\
\hphantom{\big(q_{k}^{(\alpha, \beta,\lambda)},1\big)_{\alpha,\beta}^{\lambda}}{}
= - \frac{\Gamma(\alpha+\frac{d}{2}+1)}{\Gamma(\frac{d}{2}) \Gamma(\alpha+1)}\frac{1}{2^{\alpha+\beta+1}}
\int_{-1}^1 \frac{d}{dt} P_{k}^{(\alpha, \beta)}(t) (1-t)^{\alpha} (1+t)^{\beta+1} dt \\
\hphantom{\big(q_{k}^{(\alpha, \beta,\lambda)},1\big)_{\alpha,\beta}^{\lambda}=}{}
 + \lambda \big[ a_k P_{k}^{(\alpha, \beta)}(1) - (k+\alpha+\beta+1) P_{k-1}^{(\alpha+1, \beta+1)}(1)\big].
\end{gather*}
Next, an iterated integration by parts reduces the integral in the last expression to a~beta
integral which f\/inally gives
\begin{gather*}
\int_{-1}^1 \frac{d}{dt} P_{k}^{(\alpha, \beta)}(t) (1-t)^{\alpha} (1+t)^{\beta+1} dt =
2^{\alpha+\beta+1} \frac{\Gamma(\alpha+1)\Gamma(n+\beta+1)}{\Gamma(n+\alpha+\beta+1)},
\end{gather*}
and the result follows from \eqref{eq:op1d_alpha}.
\end{proof}

The partial dif\/ferential equation \eqref{eq:Bdiff} satisf\/ied by classical
orthogonal polynomials on the ball $\{P^n_{k,\nu}(x)\}$ can be deduced from
the second-order dif\/ferential equation corresponding to the Jacobi polynomials
in the radial parts and the fact that spherical harmonics are the eigenfunctions
of the Laplace--Beltrami operator.

In the case $\mu = 1/2$, we are going to identify the polynomials in the radial
parts of $\{Q^n_{k,\nu}(x)\}$ as a family of polynomials satisfying a~fourth-order dif\/ferential
equation. Next, using polar coordinates and spherical harmonics we will obtain the explicit
expression of a fourth-order partial dif\/ferential
operator having the multivariate orthogonal polynomials as eigenfunctions.

\subsection[The case $\mu = 1/2$]{The case $\boldsymbol{\mu = 1/2}$}

In this case, the inner product reduces to
\begin{gather*}
\langle f, g \rangle_{1/2}^{\lambda}:= \frac{1}{\omega_{1/2}} \int_{\mathbb{B}^d}
 f(x) g(x) dx + \frac{\lambda}{\sigma_{d-1}} \int_{\mathbb{S}^{d-1}} f(\xi)g(\xi) d\sigma(\xi).
\end{gather*}
The mutually orthogonal basis is given by
\begin{gather*}
Q_{k,\nu}^{n}(x) = q_{k}^{(0, \beta_k,\lambda)}\big(2 \|x\|^2 -1\big) Y_\nu^{n-2k}(x),
\end{gather*}
with $1 \le k \le n/2$ and $\{Y_\nu^{n-2k}(x)\colon 1 \le \nu \le a_{n-2k}^d\}$
an orthonormal basis of spherical harmonics.

The polynomials in the radial part $q_{k}^{(0, \beta_k,\lambda)}$
are orthogonal with respect to the inner product
\begin{gather*}
(f,g)_{0,\beta_k}^{\lambda} = \frac{\sigma_{d-1}}{\omega_{1/2}}\frac{1}{2^{\beta_k+2}}
\int_{-1}^1 f(t)g(t) (1+t)^{\beta_k} dt + \lambda f(1) g(1).
\end{gather*}
Finally, using \eqref{omegamu} and \eqref{sigma_d} we get $\sigma_{d-1}/\omega_{1/2} = d$,
and writing $M = \frac{d}{2 \lambda}$ we conclude that they are orthogonal with respect to
\begin{gather*}
(f,g)_{\beta_k}^{M} = \frac{1}{2^{\beta_k+1}}
\int_{-1}^1 f(t)g(t) (1+t)^{\beta_k} dt + \frac{1}{M} f(1) g(1).
\end{gather*}
Therefore, we can recognize $q_{k}^{(0, \beta_k,\lambda)}$ as the polynomials in one of the
families studied by H.L.~Krall in 1940: the Jacobi-type polynomials.

\section{Krall's Jacobi-type orthogonal polynomials}\label{section4}

Let $\alpha = 0$, $\beta > -1$ and $M>0$. We introduce the
\textit{Jacobi-type polynomials} as follows
\begin{gather} \label{jacobi-type}
q_k^{\beta,M}(t) := \left[ M - (1+t) \dfrac{d}{dt} + k(k+\beta+1)\right] P_k^{(0,\beta)}(t),
\qquad k = 0, 1, \ldots.
\end{gather}
From Theorem~\ref{thm-3.2} we can {easily deduce} their orthogonality properties.
In fact, the polynomials $\{q_k^{\beta,M}(t)\}_{k \geqslant 0}$
are orthogonal with respect to the inner product
\begin{gather*}
(f,g)_{\beta}^{M} = \frac{1}{2^{\beta+1}}
\int_{-1}^1 f(t)g(t) (1+t)^{\beta} dt + \frac{1}{M} f(1) g(1),
\end{gather*}
and satisfy the normalization $q_k^{\beta,M}(1) = M$.

Using the dif\/ferential equation for the Jacobi polynomials $\{P_k^{(0,\beta)}(t)\}_{k \geqslant 0}$
\eqref{diffeqJ}, relation \eqref{jacobi-type} can be written as
\begin{gather} \label{oper-l1}
q_k^{\beta,M}(t) = \left[ M - \big(1-t^2\big) \dfrac{d^2}{dt^2} -(\beta+1)(1-t) \dfrac{d}{dt} \right] P_k^{(0,\beta)}(t).
\end{gather}
Therefore, both families of orthogonal polynomials are connected by means of a~second-order
linear dif\/ferential operator whose polynomial coef\/f\/icients are independent of the degree
of the polynomials.

If we denote by $d\mu_{\beta}$ the measure def\/ined by
\begin{gather*}
d\mu_{\beta} = \dfrac{1}{2^{\beta+1}}(1+t)^{\beta} dt + \dfrac{1}{M}\delta(t-1),
\end{gather*}
integrating by parts we can deduce that for arbitrary polynomials $f$ and $g$ we get
\begin{gather}
 \int_{-1}^1 g(t) \left[ M - \big(1-t^2\big) \dfrac{d^2}{dt^2} -(\beta+1)(1-t) \dfrac{d}{dt} \right] f(t)
d\mu_{\beta}(t) \label{rel-int}
\\
\qquad{}
= \dfrac{1}{2^{\beta+1}}\int_{-1}^1 f(t) \left[ M + \beta + 1- \big(1-t^2\big) \dfrac{d^2}{dt^2}
-(\beta-1-(\beta+3)t) \dfrac{d}{dt} \right] g(t) (1 + t)^{\beta} dt.\nonumber
\end{gather}

Therefore, using the orthogonality properties of $P_k^{(0,\beta)}(t)$ and $q_k^{\beta,M}(t)$, from
\eqref{oper-l1} and \eqref{rel-int} we deduce
\begin{gather}
 \left[ M + \beta + 1- \big(1-t^2\big) \dfrac{d^2}{dt^2} -(\beta-1-(\beta+3)t) \dfrac{d}{dt} \right]
q_k^{\beta,M}(t) \nonumber \\
 \qquad{} = (M + k(k+\beta))(M + (k+1)(k+\beta+1)) P_k^{(0,\beta)}(t),\label{oper-l2}
\end{gather}
where the coef\/f\/icient of $P_k^{(0,\beta)}(t)$ has been obtained identifying the
leading coef\/f\/icients in both sides of \eqref{oper-l2}.

Finally, if we combine \eqref{oper-l1} and \eqref{oper-l2} we deduce that
the polynomials $q_k^{\beta,M}(t)$ are the eigenfunctions of a~fourth-order
dif\/ferential operator with polynomials coef\/f\/icients not depending on~$k$.
The quantities $(M + k(k+\beta))(M + (k+1)(k+\beta+1))$ are the corresponding
eigenvalues.

Summarizing, we have got the following proposition.

\begin{Proposition}
Let $\mathcal{L}_1$ and $\mathcal{L}_2$ the linear operators defined by
\begin{gather*}
\mathcal{L}_1(f)(t) = \left[ M - \big(1-t^2\big) \dfrac{d^2}{dt^2} -(\beta+1)(1-t) \dfrac{d}{dt} \right] f(t),
\\
\mathcal{L}_2(f)(t) = \left[ M + \beta + 1- \big(1-t^2\big) \dfrac{d^2}{dt^2} -(\beta-1-(\beta+3)t) \dfrac{d}{dt} \right] f(t).
\end{gather*}
Then, the Jacobi-type polynomials $q_k^{\beta,M}(t)$ satisfy
\begin{gather*}
\mathcal{L}_1 P_k^{(0,\beta)}(t) = q_k^{\beta,M}(t),
\\
\mathcal{L}_2 q_k^{\beta,M}(t) = (M + k(k+\beta))(M + (k+1)(k+\beta+1)) P_k^{(0,\beta)}(t),
\end{gather*}
and the fourth-order differential equation
\begin{gather*}
\mathcal{L}_1 \mathcal{L}_2 q_k^{\beta,M}(t) = (M + k(k+\beta))(M + (k+1)(k+\beta+1))q_k^{\beta,M}(t).
\end{gather*}
\end{Proposition}

\section{A fourth-order partial dif\/ferential equation}\label{section5}

Let us return to the $d$-dimensionall ball. For $n \in \mathbb{N}_0$ and
$0 \leqslant k \leqslant n/2$, let $\{Y_\nu^{n-2k}(x)\colon 1\leqslant \nu\leqslant a_{n-2k}^d\}$
denote an orthonormal basis for $\mathcal{H}_{n-2k}^d$ and $\beta_k = n-2k + \frac{d-2}{2}$.
According to Proposition~\ref{prop-2.1}, in spherical polar coordinates $x = r \xi$ the
polynomials
\begin{gather*}
P_{k,\nu}^{n}(x) = P_{k}^{(0, \beta_k)}\big(2r^2 -1\big) r^{n-2k} Y_\nu^{n-2k}(\xi),
\end{gather*}
with $0 \leqslant k \leqslant n/2$, $1 \leqslant \nu \leqslant a_{n-2k}^d$ is a mutually
orthogonal basis of $\mathcal{V}_n^d(W_{1/2})$.
From Theorem~\ref{thm-3.1}, we deduce that the polynomials
\begin{gather} \label{baseQ0}
Q_{k,\nu}^{n}(x) = q_{k}^{\beta_k,M}\big(2r^2 -1\big) r^{n-2k} Y_\nu^{n-2k}(\xi),
\end{gather}
with $0 \leqslant k \leqslant n/2$, $1 \leqslant \nu \leqslant a_{n-2k}^d$
constitute a mutually orthogonal basis of $\mathcal{V}_n^d(W_{1/2},\frac{d}{2M})$.

First, we use the change of variable $t = 2r^2 -1$ to write relation \eqref{oper-l1} as
\begin{gather} \label{oper-lr1}
q_k^{\beta_k,M}\big(2r^2-1\big) = \left[ M - \dfrac{1-r^2}{4}\left( \dfrac{(2\beta_k+1)}{r} \dfrac{d}{dr} + \dfrac{d^2}{dr^2}\right)\right] P_k^{(0,\beta_k)}\big(2r^2-1\big),
\end{gather}
and relation \eqref{oper-l2} gives
\begin{gather}
 \left[ M + \beta_{{k}} + 1 - \dfrac{1-r^2}{4}\left( \dfrac{(2\beta_k+1)}{r} \dfrac{d}{dr} + \dfrac{d^2}{dr^2}\right) + r \dfrac{d}{dr} \right]
q_k^{\beta_k,M}\big(2r^2-1\big) \nonumber \\
 \qquad{} = (M + k(k+\beta_k))(M + (k+1)(k+\beta_k+1)) P_k^{(0,\beta_k)}\big(2r^2-1\big). \label{oper-lr2}
\end{gather}

Next, we introduce the factor $r^{n-2k}$ in \eqref{oper-lr1} and \eqref{oper-lr2}.
After a tedious but straightforward computation we obtain the linear operators
\begin{gather*}
\mathcal{M}_1 = M -\frac{1-r^2}{4} \left( \frac{d^2}{d r^2} + \frac{d-1}{r}\frac{d}{d r}
- \frac{(n-2k+d-2)(n-2k)}{r^2}\right), \\
\mathcal{M}_2 = M -\frac{1-r^2}{4} \left( \frac{d^2}{d r^2} + \frac{d-1}{r}\frac{d}{d r}
- \frac{(n-2k+d-2)(n-2k)}{r^2}\right) + \frac{d}{2} + r \frac{d}{d r},
\end{gather*}
which satisfy
\begin{gather}
\mathcal{M}_1 \big( r^{n-2k} P_k^{(0,\beta_k)}(2r^2-1)\big) = r^{n-2k} q_k^{\beta_k,M}\big(2r^2-1\big), \label{rel-M1}\\
\mathcal{M}_2 \big( r^{n-2k} q_k^{\beta_k,M}(2r^2-1)\big) = \Lambda_{n,k} r^{n-2k} P_k^{(0,\beta_k)}\big(2r^2-1\big), \label{rel-M2}
\end{gather}
where
\begin{gather*}
\Lambda_{n,k} = (M + k(k+\beta_k))(M + (k+1)(k+\beta_k+1)) \\
 \hphantom{\Lambda_{n,k}}{} = \left(M + k\left(n-k+\frac{d-2}{2}\right)\right)\left(M + (k+1)\left(n-k+\frac{d}{2}\right)\right).
\end{gather*}

\looseness=-1
Now, we can recognize the value $-(n-2k+d-2)(n-2k)$ as the eigenvalue of the Laplace--Beltrami
operator $\Delta_0$ corresponding to the spherical harmonic $Y_{\nu}^{n-2k}(\xi)$.
Therefore, if we multiply in~\eqref{rel-M1} and~\eqref{rel-M2} by $Y_{\nu}^{n-2k}(\xi)$,
relations~\eqref{L-B-operator},~\eqref{x-nabla}, and the change of variable $x = r \xi$ give
\begin{gather*}
 \left[M - \frac{1}{4}\big(1-\|x\|^2\big) \Delta\right]P_{k,\nu}^{n}(x) =Q_{k,\nu}^{n}(x), \\
 \left[M + \frac{d}{2} - \frac{1}{4}\big(1-\|x\|^2\big) \Delta + \langle x, \nabla \rangle \right]Q_{k,\nu}^{n}(x) = \Lambda_{n,k} P_{k,\nu}^{n}(x).
\end{gather*}

Finally, if we combine the previous relations we deduce that
the polynomials $Q_{k,\nu}^{n}(x)$ are the eigenfunctions of a fourth-order
partial dif\/ferential operator where the polynomial coef\/f\/icients do not depend on $n$ or $k$.
The quantities $\Lambda_{n,k}$ are the corresponding eigenvalues.

In conclusion, we have shown the following theorem.
\begin{Theorem}
Let $Q_{k,\nu}^{n}(x) = q_{k}^{\beta_k,M}(2r^2 -1) r^{n-2k} Y_\nu^{n-2k}(\xi)$
be the polynomials defined in \eqref{baseQ0} for
$0 \leqslant k \leqslant n/2$, $1 \leqslant \nu \leqslant a_{n-2k}^d$, wich
constitute a mutually orthogonal basis of $\mathcal{V}_n^d(W_{1/2},\frac{d}{2M})$.
Then, the polynomials $Q_{k,\nu}^{n}(x)$ satisfy the relations
\begin{gather*}
 \left[M - \frac{1}{4}\big(1-\|x\|^2\big) \Delta\right]P_{k,\nu}^{n}(x) =Q_{k,\nu}^{n}(x), \\
 \left[M + \frac{d}{2} - \frac{1}{4}\big(1-\|x\|^2\big) \Delta +
\langle x, \nabla \rangle \right]Q_{k,\nu}^{n}(x) = \Lambda_{n,k} P_{k,\nu}^{n}(x),
\end{gather*}
and the fourth-order partial differential equation
\begin{gather*}
\left[M - \frac{1}{4}\big(1-\|x\|^2\big) \Delta\right]\left[M + \frac{d}{2} - \frac{1}{4}\big(1-\|x\|^2\big) \Delta +
\langle x, \nabla \rangle \right]Q_{k,\nu}^{n}(x) = \Lambda_{n,k} Q_{k,\nu}^{n}(x),
\end{gather*}
where
\begin{gather*}
\Lambda_{n,k} = \left(M + k\left(n-k+\frac{d-2}{2}\right)\right)\left(M + (k+1)\left(n-k+\frac{d}{2}\right)\right).
\end{gather*}
\end{Theorem}

\section{Conclusions}\label{section6}

The central symmetry of the measure def\/ining the inner product
\eqref{eq:main-ip} is the key to
construct explicitly a sequence of mutually orthogonal polynomials
using spherical harmonics and polar coordinates. These orthogonal
polynomials depend on a family of polynomials in one variable.
The latter ones are orthogonal with respect to a modif\/ication of the measure
that is the sum of a classical Jacobi measure plus one point mass located
at the end point of the interval.

In the case $\mu = 1/2$, the polynomials in the radial
parts of $\{Q^n_{k,\nu}(x)\}$ are identif\/ied as the
Krall's Jacobi-type polynomials, which satisfy a~fourth-order dif\/ferential
equation. Next, using some properties of spherical harmonics we obtained the explicit
expression of a fourth-order partial dif\/ferential
operator having the multivariate orthogonal polynomials as eigenfunctions.
The corresponding eigenvalues depend on the polynomials indices.

In the general case, the polynomials in the radial part are a particular case of the
polynomials studied by J.~Koekoek and R.~Koekoek in~\cite{Koekoek2000}. These authors
proved that the polynomials satisfy a~linear dif\/ferential
equation of a specif\/ic form. They gave explicit expressions for the coef\/f\/icients and showed
that this dif\/ferential equation is always of inf\/inite order except if
$\mu = \alpha + 1/2$, with~$\alpha$ a~nonnegative integer, then the order is equal to
$2\alpha + 4$. Closely connected is the work by K.H.~Kwon, L.L.~Littlejohn and G.J.~Yoon
in \cite{KLY2006}.

These results suggest the possible existence of higher-order partial dif\/ferential operator
having the multivariate polynomials $\{Q^n_{k,\nu}(x)\}$ as eigenfunctions
in the case $\mu = \alpha + 1/2$, with~$\alpha$ a~nonnegative integer. This fact
deserves further research.

\subsection*{Acknowledgements}

We gratefully thanks the anonymous referees for their valuable comments and suggestions.
This work has been partially supported by DGICYT, Ministerio de Econom\'ia y Competitividad
(MINECO) of Spain and the European Regional Development Fund (ERDF) through grants MTM2011--28952--C02--02,
MTM2014--53171--P, Reseach Project P11-FQM-7276 and Research Group FQM-384 from Junta de Andaluc\'{\i}a.

\pdfbookmark[1]{References}{ref}
\LastPageEnding

\end{document}